\documentclass{article}[12pt]
\date{}

\usepackage{graphicx,tikz,color}
\usepackage{amsmath,amsfonts,amssymb,amsthm,latexsym}
\usepackage{enumerate}
\usepackage{tikz}

\newtheorem{theorem}{Theorem}[section]

\newtheorem{proposition}[theorem]{Proposition}

\newtheorem{remark}[theorem]{Remark}
\newtheorem{lemma}[theorem]{Lemma}
\newtheorem{claim}{Claim}

\newtheorem{problem}[theorem]{Problem}

\newcommand{\smallqed}{{\tiny ($\Box$)}}

\tikzstyle{vertex}=[circle, draw, inner sep=0pt, minimum size=6pt]
\newcommand{\es}{{\rm es}_{\chi}}
\newcommand{\vs}{{\rm vs}_{\chi}}
\newcommand{\ivs}{{\rm ivs}_{\chi}}

\begin{document}

\title{On the Chromatic Vertex Stability Number of Graphs}

\author{Saieed Akbari$^{a}$\thanks{Email: \texttt{s\textunderscore akbari@sharif.edu}}\and Arash Beikmohammadi$^{b}$\thanks{Email: \texttt{arash.beikmohammadi@gmail.com}}\and Sandi Klav\v zar$^{c,d,e}$\thanks{Email: \texttt{sandi.klavzar@fmf.uni-lj.si}}\and Nazanin Movarraei$^{f}$\thanks{Email: \texttt{nazanin.movarraei@gmail.com}}\smallskip}

\maketitle

\begin{center}
$^a$ Department of Mathematical Sciences, Sharif University of Technology, Iran
\medskip

$^b$ Department of Computer Engineering, Sharif University of Technology, Iran
\medskip

$^c$ Faculty of Mathematics and Physics, University of Ljubljana, Slovenia
\medskip

$^d$ Faculty of Natural Sciences and Mathematics, University of Maribor, Slovenia
\medskip

$^e$ Institute of Mathematics, Physics and Mechanics, Ljubljana, Slovenia
\medskip

$^f$ Department of Mathematics, Yazd University, Iran 
\end{center}

\begin{abstract}
The chromatic vertex (resp.\ edge) stability number ${\rm vs}_{\chi}(G)$ (resp.\ ${\rm es}_{\chi}(G)$) of a graph $G$ is the minimum number of vertices (resp.\ edges) whose deletion results in a graph $H$ with $\chi(H)=\chi(G)-1$. In the main result it is proved that if $G$ is a graph with $\chi(G) \in \{ \Delta(G), \Delta(G)+1 \}$, then ${\rm vs}_{\chi}(G) = {\rm ivs}_{\chi}(G)$, where ${\rm ivs}_{\chi}(G)$ is the independent chromatic vertex stability number. The result need not hold for graphs $G$ with $\chi(G) \le \frac{\Delta(G)+1}{2}$. It is  proved that if $\chi(G) > \frac{\Delta(G)}{2}+1$, then ${\rm vs}_{\chi}(G) = {\rm es}_{\chi}(G)$. A Nordhaus-Gaddum-type result on the chromatic vertex stability number is also given. 
\end{abstract}


\noindent
{\bf MSC Subj.\ Class.\ (2010)}: 05C15

\baselineskip16pt

\section{Introduction}

Throughout this paper all graphs are finite, simple, and having at least one edge. The chromatic edge stability number ${\rm es}_{\chi}(G)$ of a graph $G$ is the minimum number of edges whose deletion results in a graph $H$ with $\chi(H)=\chi(G)-1$. This natural coloring concept was introduced in 1980 by Staton~\cite{staton-1980}, and independently rediscovered much later in~\cite{arumugam-2008}. Nevertheless, this concept has become the subject of wider interest only recently. The paper~\cite{kemnitz-2018} gives, among other results, a characterization of graphs with equal chromatic edge stability number and chromatic bondage number. In~\cite{bresar-2020}, edge-stability critical graphs were studied. The paper~\cite{akbari-2020} brings Nordhaus-Gaddum type inequality for $\es(G)$ (stronger than a related result from~\cite{arumugam-2008}), sharp upper bounds on $\es(G)$ in terms of size and of maximum degree, and a characterization of graphs with $\es(G) = 1$ among $k$-regular graphs for $k\le 5$. In~\cite{huang-2021+} progress on three open problems from~\cite{akbari-2020} are reported. The chromatic edge stability number has been generalized to arbitrary graphical invariants in~\cite{kemnitz-2021+}, where in particular it was considered with respect to the chromatic index, see also~\cite{akbari-2021+, alikhani-2020+}. 

Like edge stability numbers, vertex stability numbers were introduced in the 1980s or earlier. In~\cite{Bauer-1983}, the $\mu$-stability of a graph $G$, where $\mu$ is an arbitrary graph invariant, is defined as the minimum number of vertices whose removal changes $\mu$. The paper~\cite{Bauer-1983} then proceeds by investigating the stability with respect to the domination number and the independence number, which in turn led to a series papers investigation the stability with respect to these two invariants. In this paper, however, we are interested in the stability with respect to the chromatic number. At least as far as we know, this concept has not yet been explored (in~\cite{Bauer-1983},  the stability with respect to the chromatic number is briefly mentioned only in one sentence) which we find quite surprising since vertex versions are usually considered before edge versions. The closest investigation we are aware of is the paper~\cite{alikhani-2018}, where the stability with respect to the distinguishing number is investigated. 

Let $G$ be a graph. The {\emph{chromatic vertex stability number}} $\vs(G)$ of $G$ is the minimum number of vertices of $G$ such that their deletion results in a graph $H$ with $\chi(H)=\chi(G)-1$. For instance, it is straightforward to see that  $\vs(P)=3$, where $P$ is the Petersen graph. Note that if $\chi(G) = 3$, then $\vs(G)$ is just the minimum cardinality of a set $X\subseteq V(G)$ such that the graph induced by $V(G)\setminus X$ is bipartite. 

We also introduce the {\emph{independent chromatic vertex stability number}}, $\ivs(G)$, of $G$ as the minimum number of independent vertices such that their deletion results in a graph $H$ with $\chi(H)=\chi(G)-1$. Then our main result reads as follows. 

\begin{theorem}
\label{thm:main}
If $G$ is a graph with $\chi(G) \in \{ \Delta(G), \Delta(G)+1 \}$, then $\vs(G)=\ivs(G)$.
\end{theorem}

The paper is structured as follows. In the rest of this section we recall needed definitions and concepts. In Section~\ref{sec:proof} we prove Theorem~\ref{thm:main}. In the subsequent section we show that Theorem~\ref{thm:main} need not hold for graphs $G$ with $\chi(G) \le \frac{\Delta(G)+1}{2}$, and discuss a possible threshold function $f(\Delta(G))$ that would guarantee that if $\chi(G) \ge f(\Delta(G))$, then $\vs(G)=\ivs(G)$. In the final section we prove that if $\chi(G) > \frac{\Delta(G)}{2}+1$, then $\vs(G)=\es(G)$, and give a  Nordhaus-Gaddum-type result on the chromatic vertex stability number. 

Given a graph $G=(V(G), E(G))$, a function $c:V(G) \to [k] = \{1,\ldots , k\}$ with $c(v) \neq c(u)$ for each edge $uv$ is a {\emph{proper $k$-coloring}} of $G$. The minimum $k$ for which $G$ admits a proper $k$-coloring is the {\emph{chromatic number}} $\chi(G)$ of $G$. If $c$ is a proper coloring of $G$, then the set of all vertices of $G$ with color $i$, $i \in [\chi(G)]$, is a {\em color class} and will be  denoted by $C_i$. The \emph{open neighborhood} of a vertex $v$ in $G$ is the set of neighbors of $v$, denoted by $N_G(v)$, whereas the \emph{closed neighborhood} of $v$ is $N_G[v] = N_G(v) \cup \{v\}$. The \emph{degree} of a vertex $v$ in $G$ is denoted by $d_G(v)$. The subgraph of $G$ induced by $A \subseteq V(G)$ will be denoted by $G \left[ A \right]$. The complete graph of order $n$ is denoted by $K_n$ and the complement of a graph $G$ by $\overline{G}$. Finally, the order of $G$ will be denoted by $n(G)$.

\section{Proof of Theorem~\ref{thm:main}}
\label{sec:proof}

The following lemma follows directly from Brooks' Theorem (cf.~\cite[p.197]{west-2001}).

\begin{lemma}\label{lemma:brooks}
Let $G$ be a connected graph with $\chi(G)=\Delta(G)+1$. If $\Delta(G) \neq 2$, then $G \cong K_{\Delta(G)+1}$, and if $\Delta(G)=2$, then $G \cong C_n$ for some odd $n$.
\end{lemma}

For the proof of the theorem, we also need the following lemma. 

\begin{lemma}\label{lemma:delta}
Let $G$ be a connected graph with $\chi(G)=\Delta(G)$ and $\vs(G)=1$. Then there exists $v \in V(G)$ such that $d_G(v)=\Delta(G)$ and $\chi(G-v)=\Delta(G)-1$.  
\end{lemma}

\begin{proof}
Let $S = \{u\in V(G):\ \chi(G-u)=\Delta(G)-1\}$. Since $\vs(G)=1$, we have $S \neq \emptyset$. If $S=V(G)$, then there exists $u \in S$ such that $d_G(u)=\Delta(G)$, as desired. Otherwise, $S \neq V(G)$ and by the connectivity of $G$ there exists $xy \in E(G)$ such that $x \in S$ and $y \in V(G) \setminus S$. Let $c$ be a proper $(\Delta(G)-1)$-coloring of $G-x$. We now claim that $d_G(x) = \Delta(G)$. By contradiction, assume that $d_G(x)<\Delta(G)$ and consider the following two cases. 

\medskip\noindent
{\bf{Case 1.}} $d_G(x)<\Delta(G)-1$. \\
There is a color class $C_i$ such that $i \in [\Delta(G)-1]$ and $N_G(x) \cap C_i = \emptyset$. So, we can color $x$ by $i$ to obtain a proper $(\Delta(G)-1)$-coloring of $G$, a contradiction.

\medskip\noindent
{\bf{Case 2.}} $d_G(x) = \Delta(G)-1$. \\
We may assume that $|N_G(x) \cap C_i|=1$ for each $i \in [\Delta(G)-1]$, for otherwise we may proceed as in Case~1 to get a contradiction. But then we can color $x$ by $c(y)$ to obtain a proper $(\Delta(G)-1)$-coloring of $G-y$. Hence, $y \in S$, a final contradiction.
\end{proof}

Note that in Lemma~\ref{lemma:delta}, the only graph that applies to the case $\chi(G)=\Delta(G)=2$ and $\vs(G)=1$ is $P_3$. Also note that the proof of Lemma~\ref{lemma:delta} asserts that every vertex in $S$ that has a neighbor in $V(G)\setminus S$ has maximum degree. We now proceed with the proof of the theorem.

{\noindent \it Proof of Theorem~\ref{thm:main}.} If $\chi(G)=\Delta(G)+1$, then $\vs(G)$ is the number of connected components $C$ of $G$ with $\chi(C)=\Delta(G)+1$ which are complete graphs or odd cycles by Lemma~\ref{lemma:brooks} for which $\vs(C)=1$ holds. We can remove one vertex from each of these components to reduce the chromatic number of $G$ by $1$. The set of these removed vertices is an independent set which means that $\vs(G)=\ivs(G)$, as desired.

Now, let $\chi(G)=\Delta(G)$. If $\Delta(G) \leq 2$, then it is easy to see that $\vs(G)=\ivs(G)$. Hence we may assume in the rest of the proof that $\Delta(G) \geq 3$. We proceed by induction on $\vs(G)$. Clearly, the assertion holds for $\vs(G)=1$.

Let $S \subseteq V(G)$ be a set of vertices such that $|S|=\vs(G) \geq 2$, $\chi(G \setminus S)=\Delta(G) -1$, and $|E(G[S])|$ is as small as possible. We are going to prove that $S$ is an independent set. By contradiction, suppose that $S$ is not an independent set. So, there are $u,v \in S$ such that $uv \in E(G)$. Set $G'=G \setminus S=G[V(G) \setminus S]$ and let $c$ be a proper $(\Delta(G)-1)$-coloring of $G'$.

\begin{claim}\label{claim:S}
Let $w \in S$. Then $N_G(w) \cap C_i \neq \emptyset$, for each $i \in [\Delta(G)-1]$.
\end{claim}
By contradiction, suppose that there is a color class $C_i$ such that $N_G(w) \cap C_i = \emptyset$. Now, we can color $w$ by $i$ to obtain a proper $(\Delta(G)-1)$-coloring of $G \setminus S'$, where $S'=S \setminus \{ w \}$. So, $\vs(G) \leq |S'| = |S|-1$, a contradiction. \smallqed

\begin{claim}\label{claim:S3}
$|N_G(u) \cap C_i| = 1$, $|N_G(v) \cap C_i| = 1$, for each $i \in [\Delta(G)-1]$.
\end{claim}
By Claim~\ref{claim:S}, $|N_G(u) \cap C_i| \geq 1$, $|N_G(v) \cap C_i| \geq 1$, for each $i \in [\Delta(G)-1]$. If $|N_G(u) \cap C_j| > 1$ for some $j \in [\Delta(G)-1]$, then because $uv \in E(G)$, we get $d_G(u) > \Delta(G)$, a contradiction. \smallqed

\begin{claim}\label{claim:S2}
Let $w \in S$. Then $|N_G(w) \cap C_i| \leq 2$, for each $i \in [\Delta(G)-1]$.
\end{claim}
By Claim~\ref{claim:S}, $|N_G(w) \cap C_i| \geq 1$, for each $i \in [\Delta(G)-1]$. If $|N_G(w) \cap C_j| > 2$ for some $j \in [\Delta(G)-1]$, then we can conclude that $d_G(w) > \Delta(G)$, a contradiction. \smallqed

\begin{claim}\label{claim:degreeInS}
$\Delta(G[S]) \leq 1$.
\end{claim}
By Claim~\ref{claim:S}, if $w \in S$, then $|V(G') \cap N_G(w)| \geq \Delta(G)-1$. So, $d_{G[S]}(w) \leq 1$.  \smallqed 

\begin{claim}\label{claim:S1}
$E(G[S])=\{ uv \}$.
\end{claim}

Since $\vs(G) \geq 2$ we have $\chi(G-u)=\Delta(G)$. Moreover, $\chi(G-u) \leq \Delta(G-u)+1$, so $\Delta(G) \in \{\Delta(G-u), \Delta(G-u)+1\}$. Also, $\vs(G-u)=|S|-1 \geq 1$.
By induction, there is an independent set $S' \subseteq V(G) \setminus \{ u \}$ such that $\chi((G - u)\setminus S') = \Delta(G)-1$ and $|S'|=|S|-1$. Since $S'$ is an independent set, similarly as in Claim~\ref{claim:degreeInS}, one can deduce that $\Delta(G[S' \cup \{ u \}]) \leq 1$. Thus, $|E(G[S' \cup \{ u \}])| \leq 1$ and we get $1 \leq |E(G[S])| \leq |E(G[S' \cup \{ u \}])| \leq 1$, where the middle inequality follows by the minimality assumption on the number of edges in $G[S]$. Hence, $E(G[S])=\{ uv \}$, as desired.  \smallqed

\medskip
Now, let $H=G[V(G') \cup \{ u \}]$ and let $H_u$ be the connected component of $H$ containing $u$. Clearly, if $H$ is connected, then $H = H_u$. Similarly, let $H'=G[V(G') \cup \{ v \}]$ and let $H_v$ be the connected component of $H'$ containing $v$. Note that $\chi(H)=\chi(H')=\Delta(G)$.

\begin{claim}\label{claim:complete}
If $\Delta(G) > 3$, then $H_u \cong K_{\Delta(G)}$ and $H_v \cong K_{\Delta(G)}$. If $\Delta(G)=3$, then $H_u \cong C_{2a+1}$ and $H_v \cong C_{2b+1}$ for some natural numbers $a$ and $b$. 
\end{claim}

Each connected component of $H$ except $H_u$ has chromatic number at most $\Delta(G)-1$ because it is a subgraph of $G'$, whose chromatic number is $\Delta(G) -1$. Since $\chi(H)=\Delta(G)$, we can conclude that $\chi(H_u) = \Delta(G) \leq \Delta(H_u)+1$, which implies $\Delta(H_u) \geq \Delta(G)-1$.
If $\Delta(H_u)= \Delta(G)-1$, then by Lemma~\ref{lemma:brooks}, $H_u \cong K_{\Delta(H_u)+1}=K_{\Delta(G)}$ if $\Delta(H_u)=\Delta(G)-1>2$ and $H_u \cong C_{2a+1}$ if $\Delta(H_u)=\Delta(G)-1=2$, as desired.
Otherwise, $\Delta(H_u)=\Delta(G)$ and by Lemma~\ref{lemma:delta}, there is a vertex $x \in V(H_u)$ such that $d_{H_u}(x)=\Delta(G)$ and $\chi(G \setminus S') < \Delta(G)$, where $S'=(S \setminus \{ u \}) \cup \{ x \}$.
It is easy to see that $S'$ is an independent set because $S \setminus \{ u \}$ is an independent set and all neighbors of $x$ are in $V(H_u)$ and not in $S \setminus \{ u \}$ (because $d_{H_u}(x)=\Delta(G)$ and $V(H_u) \cap (S \setminus \{ u \}) = \emptyset$).
Hence $|E(G[S'])| < |E(G(S])|$, a contradiction.

By the same method, $H_v \cong K_{\Delta(H_v)+1}=K_{\Delta(G)}$ for $\Delta(H_v)=\Delta(G)-1>2$ and $H_v \cong C_{2b+1}$ for $\Delta(H_v)=\Delta(G)-1=2$.  \smallqed

\begin{claim}\label{claim:cap}
$N_G(u) \cap N_G(v) = \emptyset$.
\end{claim}

\noindent
{\bf{Case 1.}} $\Delta(G) > 3$. \\
By Claim~\ref{claim:complete}, $H_u \cong H_v \cong K_{\Delta(G)}$.
By contradiction, suppose that $N_G(u) \cap N_G(v) \neq \emptyset$ and let $x \in N_G(u) \cap N_G(v)$. Clearly, $x \notin S$ and $x \in V(H_u) \cap V(H_v)$. Therefore, $x$ is adjacent to all vertices of $N_G[u] \cup N_G[v]$. If $N_G[u] \neq N_G[v]$, then $d_G(x) > \Delta(G)$, a contradiction. Otherwise, $N_G[u] = N_G[v]$ and thus $G[N_G[u]] \cong K_{\Delta(G)+1}$ which means $\chi(G)=\Delta(G)+1$, a contradiction.

\medskip\noindent
{\bf{Case 2.}} $\Delta(G) = 3$. \\
By Claim~\ref{claim:complete}, $H_u \cong C_{2a+1}$ and $H_v \cong C_{2b+1}$ for some natural numbers $a$ and $b$. By contradiction, suppose that $N_G(u) \cap N_G(v) \neq \emptyset$. Since $N_G(u) \cap N_G(v) \neq \emptyset$, $H_u \cong C_{2a+1}$, $H_v \cong C_{2b+1}$, and $H_u$ and $H_v$ are connected components of $H$ and $H'$, we can conclude that $H_u -u = H_v -v$. So, $u$ and $v$ have two common neighbors. Let $x$ and $y$ be those vertices. If $2a+1=2b+1=3$, then $u$, $v$, $x$, and $y$ form a $K_4$, which means $\chi(G) > \Delta(G) = 3$, a contradiction. Otherwise, $x$ and $y$ are not adjacent and since $c(x) \neq c(y)$ by Claim~\ref{claim:S3}, we can color $u$ and $v$ by $c(x)$ and $c(y)$, respectively, to obtain a proper $(\Delta(G)-1)$-coloring of $G \setminus S'$, where $S'=(S \setminus  \{ u,v \})) \cup \{ x,y \}$. Since $d(x) \leq \Delta(G) = 3$ and $d(y) \leq \Delta(G) =3$, we can conclude that $x$ and $y$ do not have neighbors in $S'$ which means $S'$ is an independent set, a contradiction.  \smallqed

\medskip
Now, among all proper $(\Delta(G)-1)$-coloring of $G'$ let $c'$ be selected such that its smallest color class is as small as possible. We may assume without loss of generality that $C'_1$ is such a color class.

\begin{claim}\label{claim:colorClass}
If $x \in C'_1$, then $|N_G(x) \cap S| \leq 2$.
\end{claim}
Indeed, otherwise there exists $i \in [\Delta(G)-1] \setminus \{ 1 \}$ such that $N_G(x) \cap C'_i = \emptyset$. So, we can color $x$ by $i$ to obtain a proper $(\Delta(G)-1)$-coloring of $G'$. In this new coloring, the number of vertices colored by $1$ is $|C'_1|-1$. Therefore, we can deduce that $|C'_1|$ is not a smallest possible color class, a contradiction.  \smallqed

\medskip
Now, let $P$ be the connected component of $G[S \cup C'_1]$ containing $u$ and $v$.
By Claims~\ref{claim:S3} and~\ref{claim:S1}, $d_P(u)=d_P(v)=2$. By Claims~\ref{claim:S2} and~\ref{claim:S1}, $d_P(w) \leq 2$, for each $w \in S \cap V(P)$.
By Claim~\ref{claim:colorClass}, $d_P(x) \leq 2$, for each $x \in C'_1 \cap V(P)$. Hence, $\Delta(P) = 2$, which means $P$ is a path or a cycle.
Now, we have two cases. 

\medskip\noindent
{\bf{Case 1.}} $P$ is a path. \\
Let $P=x_k, \ldots, x_1, u, v, y_1, \ldots, y_l$ and let $X_1= \{x_1, x_3, \ldots \}$ and $X_2= \{x_2, x_4, \ldots \}$. Clearly, $X_1 \subseteq C'_1$ and $X_2 \subseteq S$. Now, we can color all vertices of $X_2 \cup \{ u \}$ by $1$ to obtain a proper $(\Delta(G)-1)$-coloring of $G \setminus S'$, where $S'=(S\setminus (X_2 \cup \{ u \})) \cup X_1$. Clearly, $|S'| \leq |S|$, and $S'$ is an independent set since $X_1$ and $S \setminus (X_2 \cup \{ u \})$ are independent sets and by Claim~\ref{claim:colorClass} there is no edge between $X_1$ and $S' \setminus X_1$. Hence, $0=|E(G[S'])| < |E(G[S])|=1$, a contradiction.

\medskip\noindent
{\bf{Case 2.}} $P$ is a cycle.\\
It is easy to see that $P$ is an odd cycle.
According to Claim~\ref{claim:complete}, $H_u \cong K_{\Delta(G)}$ or $H_u \cong C_{2a+1}$ for a natural number $a$. By Claim~\ref{claim:cap}, since $u$ and $v$ do not have a common neighbor, we can conclude that $|V(P)| > 3$.
Now, we have two subcases. 

\medskip\noindent
{\bf{Subcase 2.1.}} $H_u \cong K_{\Delta(G)}$. \\
In this subcase, $\Delta(G) > 3$, or $\Delta(G)=3$ and $H_u \cong K_3$.
Let $P=u, v, x_1, x_2, \ldots , x_{2k-1}, u$, where $k$ is a natural number greater than 1, and let $X_1= \{ x_1,x_3, \ldots, x_{2k-1} \}$ and $X_2= \{ x_2,x_4, \ldots, x_{2k-2} \}$. Clearly, $X_1 \subseteq C'_1$ and $X_2 \subseteq S$.
Since $X_2 \cup \{ v \}$ is an independent set, we can color all vertices of $X_2 \cup \{ v \}$ by $1$ to obtain a proper $(\Delta(G)-1)$-coloring of $G \setminus S'$, where $S'=(S \setminus (X_2 \cup \{ v \})) \cup X_1 $. Clearly, $|S'|=|S|-|X_2 \cup \{ v \}|+|X_1|=|S|$. It is easy to see that there is only one edge $ux_{2k-1}$ between $X_1$ and $S \setminus (X_2 \cup \{ v \})$. In addition, $X_1$ and $S \setminus (X_2 \cup \{ v \})$ are independent sets; hence, $E(G[S'])=\{ ux_{2k-1} \}$. Clearly, $x_{2k-1} \in V(H_u)$. Since $H_u \cong K_{\Delta(G)}$ and $\Delta(G) \geq 3$, there is a vertex in $V(H_u)$ adjacent to both $u$ and $x_{2k-1}$ which means $N_G(u) \cap N_G(x_{2k-1}) \neq \emptyset$. Since $|S'|=|S|$ and $|E(G[S'])|=1$, by the same method as in the proof of Claim~\ref{claim:cap}, we can show that $N_G(u) \cap N_G(x_{2k-1}) = \emptyset$, a contradiction.

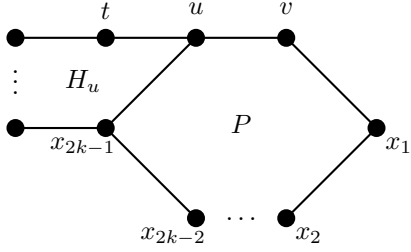
\begin{figure}[ht!]
\begin{center}
\begin{tikzpicture}[scale=1.0,style=thick,x=0.6cm,y=0.6cm]
\def\vr{3pt}
\coordinate(u) at (4,6);
\coordinate(v) at (6,6);
\coordinate(x_{2k-1}) at (2,4);
\coordinate(x_1) at (8,4);
\coordinate(x_{2k-2}) at (4,2);
\coordinate(x_2) at (6,2);
\coordinate(t) at (2,6);
\coordinate(t') at (0,6);
\coordinate(t'') at (0,4);


\draw (u) -- (v);
\draw (u) -- (x_{2k-1});
\draw (v) -- (x_1);
\draw (x_{2k-1}) -- (x_{2k-2});
\draw (x_1) -- (x_2);
\draw (u) -- (t);
\draw (t') -- (t);
\draw (t'') -- (x_{2k-1});

\draw (4,7) node[below] {$u$};
\draw(6,7) node[below] {$v$};
\draw(1.5,4) node[below] {$x_{2k-1}$};
\draw(3.5,2) node[below] {$x_{2k-2}$};
\draw (2,7) node[below] {$t$};
\draw(8.5,4) node[below] {$x_{1}$};
\draw(6.5,2) node[below] {$x_{2}$};

\draw(5,4.5) node[below] {$P$};
\draw(1.5,5.5) node[below] {$H_u$};

\draw(u)[fill=black] circle(\vr);
\draw(v)[fill=black] circle(\vr);
\draw(x_{2k-1})[fill=black] circle(\vr);
\draw(x_1)[fill=black] circle(\vr);
\draw(x_{2k-2})[fill=black] circle(\vr);
\draw(x_2)[fill=black] circle(\vr);
\draw(t)[fill=black] circle(\vr);
\draw(t')[fill=black] circle(\vr);
\draw(t'')[fill=black] circle(\vr);

\node[draw=none] at (5,2) {$\ldots$};
\node[draw=none] at (0,5.2) {$\vdots$};

\end{tikzpicture}
\end{center}
\caption{$\Delta(G)=3$, and $P$ and $H_u$ are odd cycles}
\label{fig:delta}
\end{figure}

\medskip\noindent
{\bf{Subcase 2.2.}} $H_u=C_{2r+1}$ for some $r>1$. \\
In this subcase $\Delta(G)=3$. With no loss of generality, assume that $c'$, the coloring described above, is a coloring where $|V(P)|$ is as large as possible. We have $\{ u,x_{2k-1} \} \subseteq V(P) \cap V(H_u)$, cf.~Fig.~\ref{fig:delta}. Moreover, it is also easy to see that $V(P) \cap V(H_u)=\{ u,x_{2k-1} \}$ because if $w \in (V(P) \cap V(H_u)) \setminus \{ u, x_{2k-1} \}$, then $c'(w)=1$ and $w$ has two neighbors with color $2$ and $w$ has two neighbors in $S$ which contradicts $\Delta(G)=3$. Now, let $t$ be the other neighbor of $u$ in $H_u$ in the proper $2$-coloring of $G'$. Clearly, $c'(x_{2k-1})=1$ and $c'(t)=2$. Let $c''$ be a proper $(\Delta(G)-1)$-coloring ($2$-coloring) of $G'$ switching the color of vertices in $H_u -u$ (keeping the color of other vertices the same as $c'$). Claim~\ref{claim:colorClass} still holds for $c''$ because $|C'_1|=|C''_1|$. Now, let $P'$ be the connected component of $G[S \cup C''_1]$ containing $u$ and $v$. If $P'$ is a path, then simillar to the Case 1, we are done. Otherwise, $P'$ is a cycle. Obviously, $V(P) \setminus \{ x_{2k-1} \} \subseteq V(P')$ and $t \in V(P')$, and since $|V(P)| \geq |V(P')|$ and $P'$ is a cycle, we can conclude that $x_{2k-2}t \in E(G)$. Now, one can color $u$ and $x_{2k-2}$ by $2$ to obtain a proper $2$-coloring of $G \setminus S''$, where $S''=(S \setminus \{ u,x_{2k-2} \} ) \cup \{ x_{2k-1} \}$. Since $|S''| < |S|$, we have a final contradiction. \qed

\section{Discussion on graphs with smaller chromatic number}
\label{sec:discuss}

In the next remark we demonstrate that Theorem~\ref{thm:main} does not extend to the case when $\chi(G) \le \frac{\Delta(G)}{2}$.

\begin{remark}
\label{rem:1}
\normalfont
Let $G_{n,k}$, $n\ge 2$, $k\ge 3$, be the graph obtained from the cycle $C_{2n}$ as follows. For each vertex $u$ of the cycle take two disjoint complete graphs $K_k$, select a fixed vertex in each of them, and identify the two vertices with $u$. Then it is straightforward to verify that  $\chi(G_{n,k}) = k$, $\Delta(G_{n,k}) = 2k$, $\vs(G_{n,k}) = 2n$, and $\ivs(G_{n,k}) = 3n$. 
\end{remark}

In view of Remark~\ref{rem:1} we pose the next problem for which we feel the answer is yes. 

\begin{problem}\label{problem1}
Is it true that if $G$ is a graph with $\chi(G) \geq \frac{\Delta(G)}{2}+1$, then $\vs(G)=\ivs(G)$?
\end{problem}

Remark~\ref{rem:1} demonstrates that in Problem~\ref{problem1} the assumption $\chi(G) \geq \frac{\Delta(G)}{2}+1$ cannot be weakened to $\chi(G) \geq \frac{\Delta(G)}{2}$. We next demonstrate that furthermore the assumption cannot be weakened to $\chi(G) \geq \frac{\Delta(G)+1}{2}$. 

 Let $k \geq 3$ be an odd positive number and let $H_k$ be the graph obtained from two copies of $K_{\frac{k+1}{2}}$ by identifying a vertex from each of the copies.  Let $G_k$ be the graph obtained from two disjoint copies of $H_k$ by adding an edge between the maximum degree vertices in the two copies. Then we have $\Delta(G_k) = k$, $\chi(G_k) = \frac{k+1}{2} = \frac{\Delta(G_k)+1}{2}$, and $2 = \vs(G_k) < \ivs(G_k) = 3$.

Using a similar construction we next show that the ratio 
$$\frac{\ivs(G_r)}{\vs(G_r)}$$
can be arbitrary large. Let $k$ be a natural number and let $H'_k$ be the graph obtained from $k$ copies of $K_{3}$ by identifying a vertex from each of the copies. Let further $G'_k$ be the graph obtained from two disjoint copies of $H'_k$ by adding an edge between the two vertices of maximum degrees in the copies. Then $\chi(G_k') = 3$, $\vs(G_k')=2$, and $\ivs(G_k')=k+1$. This example shows that for any natural number $r$, there exists a graph $G_r$ such that $\frac{\ivs(G_r)}{\vs(G_r)} > r$.

\section{Relation with $\es$ and a Nordhaus-Gaddum-type result}\label{sec:es3}
\label{sec:more}

In this final section we first prove that if the chromatic number of a graph is large, then its chromatic vertex stability number is equal to its chromatic edge stability number. More precisely, we have the following result. 

\begin{theorem}\label{theorem:chromatic_vertex_stability_index}
If $G$ is a graph with $\chi(G) > \frac{\Delta(G)}{2}+1$, then $\vs(G)=\es(G)$.
\end{theorem}

\begin{proof}
Clearly, $\vs(G) \leq \es(G)$, hence we need to show that $\vs(G) \geq \es(G)$. We proceed by induction on $\vs(G)$. 

The base case is when $\vs(G)=1$. Let $u$ be a vertex such that $\chi(G - u) = \chi(G)-1$. Let $c$ be a proper $(\chi(G)-1)$-coloring of $G - u$. There is a color class $C_i$ such that $|N_G(u) \cap C_i| \leq 1$, because otherwise $d_G(u) \geq 2(\chi(G)-1) > \Delta(G)$, a contradiction. If $N_G(u) \cap C_i=\emptyset$ for some color class $C_i$, then we can color $u$ by $i$ to obtain a proper $(\chi(G)-1)$-coloring of $G$, a contradiction. Otherwise, $|N_G(u) \cap C_i|=1$ and let $e$ be the edge between $u$ and $C_i$. So, we can color $u$ by $i$ to obtain a proper $(\chi(G)-1)$-coloring of $G-e$. Thus, $1=\vs(G) \geq \es(G)$, and the base case is proved. 

Now, suppose $\vs(G)=k > 1$ and let $S=\{u_1,\ldots ,u_k\} \subseteq V(G)$ be a subset such that $\chi(G \setminus S) = \chi(G)-1$. Let $c$ be a proper $(\chi(G)-1)$-coloring of $G \setminus S$. Similar to the base case of the induction, there is a color class $C_i$ such that $|N_G(u_1) \cap C_i|=1$. Now, suppose that $e$ is the edge between $u_1$ and $C_i$ and define $G'=G-e$. Clearly, $\vs(G')=k-1$. By induction, $\vs(G') \geq \es(G')$ and thus $\es(G) \leq \es(G')+|\{ e \}| \leq k=\vs(G)$, which proves the induction step.
\end{proof}

The lower bound given in Theorem~\ref{theorem:chromatic_vertex_stability_index} is sharp. Let $k$ be an even positive integer and $G$ be the graph obtained from two copies of $K_{\frac{k}{2}+1}$ which have only one vertex in common. Note that $\Delta(G) = k$, $\chi(G) = \frac{k}{2}+1 = \frac{\Delta(G)}{2}+1$, and $\vs(G)=1$, $\es(G)=2$.

To conclude the paper, we give the following Nordhaus-Gaddum-type result. 

\begin{proposition}\label{lemma:vs_upperbound2}
If $G$ is a not edgeless and not complete graph, then $\vs(G) + \vs(\overline{G}) \leq n(G)+1$.
\end{proposition}

\begin{proof}
Considering a proper $\chi(G)$-coloring of $G$ and its smallest color class, we see that $\vs(G) \leq \frac{n(G)}{\chi(G)}$. The clasical Nordhaus-Gaddum inequalities~\cite{Nord-1956} assert that $\chi(G) + \chi(\overline{G}) \le n(G) + 1$ and $\chi(G) \chi(\overline{G}) \ge n(G)$, hence we can estimate as follows: 
$$\vs(G) + \vs(\overline{G}) \leq \frac{n(G)}{\chi(G)} + \frac{n(G)}{\chi(\overline{G})} =  n(G) \frac{\chi(G) + \chi(\overline{G})}{\chi(G) \chi(\overline{G})} \le n(G) \frac{n(G)+1}{n(G)} = n(G) + 1$$
and we are done.
\end{proof}

\section*{Acknowledgments}

We thank one of the reviewers for careful reading of the article and for many very helpful comments.

The research of the first author was supported by grant number G981202 from Sharif University of Technology. Sandi Klav\v{z}ar acknowledges the financial support from the Slovenian Research Agency (research core funding P1-0297 and projects N1-0095, J1-1693, J1-2452).

\end{document}